\documentclass[a4paper,10pt,reqno]{amsart}
\usepackage{graphicx}
\usepackage{a4wide}
\usepackage{amssymb}
\usepackage{amsthm}
\usepackage{eucal}
\usepackage{color}
\usepackage[pdftex,colorlinks]{hyperref}

\usepackage{mathabx} 
\numberwithin{equation}{section}

\newcommand{\R}{\mathbb R}

\newcommand{\ii}{\int_{\Omega} } 

\newcommand{\be}{\begin{equation}}
\newcommand{\ee}{\end{equation}}
\newcommand{\ba}{\begin{eqnarray}}
\newcommand{\ea}{\end{eqnarray}}

\newcommand{\eps}{\epsilon}

\newcommand{\Om}{\Omega}
\newcommand{\om}{\omega}
\newcommand{\diver}{\mathbf{div}}

\def\dis{\displaystyle}

\newtheorem{theorem}{Theorem}[section]
\newtheorem{proposition}[theorem]{Proposition}
\newtheorem{remark}[theorem]{Remark}

\newtheorem{lemma}[theorem]{Lemma}

\newtheorem{definition}[theorem]{Definition}

\newtheorem{assumption}[theorem]{Assumption}
\newtheorem{claim}[]{Claim}

\title[Controllability for equations of Sobolev-Galpern type]{On the controllability of some equations of Sobolev-Galpern type}

\author{\textsc{Felipe W. Chaves-Silva}}
\thanks{Department of Mathematics, Federal University of Pernambuco, UFPE, CEP 50740-545, Recife, PE, Brazil. 
E-mail: {\tt fchaves@dmat.ufpe.br.}
F. W. Chaves-Silva ERC project number 32084: Semi Classical Analysis
of Partial Differential Equations, ERC-2012-ADG}
\author{\textsc{Diego A. Souza}}
\thanks{ Department of Mathematics, Federal University of Pernambuco, UFPE, CEP 50740-545, Recife, PE, Brazil. 
E-mail: {\tt diego.souza@dmat.ufpe.br}. 
D. A. Souza was partially supported by the ERC advanced grant 668998 (OCLOC) under the EU's H2020 research program.}
\date{}

\begin{document}
\maketitle


\begin{abstract}
 In this paper we deal with the controllability problem for some Sobolev type equations. We show that the equations cannot be driven to zero if the control region is strictly supported within the domain. Nevertheless, we also prove that it is possible to control the equations using controls which have a moving support,  under some assumptions on its movement.
\end{abstract}

\vskip 0.25cm


\tableofcontents


\section{Introduction}

	Let  $\Omega \subset \mathbb{R}^N$ $(N\geq 1)$ be a bounded connected open set 
	whose boundary $\partial \Omega$ is regular enough.  Let $T > 0$ and $\mathcal{O}$  
	be a nonempty subset of $\Omega\times(0,T)$. We will use the notation 
	$Q =  \Omega \times (0,T)$ and  $\Sigma = \partial \Omega \times (0,T)$.

 	In this paper we deal with controllability properties for some \textit{pseudo-parabolic equations} 
	of the form 
\be\label{generalSobGal}
	(I- \gamma \mathcal{L})\partial_ty +\mathcal{M}y= f,
\ee
 	where $\gamma$ is positive real number and  $\mathcal{L}$ and $\mathcal{M}$  are linear partial 
	differential operators in the spatial variable of order $2l$ and $m$ with $ m \leq 2l$, respectively (\cite{lagnese,Showalter,Showalterbook,ST}). 
	More precisely, 
	we  consider the following two problems
\begin{equation}\label{eq:BZK}
	\left|   
		\begin{array}{lcl}
			y_t  - \Delta y_t  - \Delta y   = v\chi_{\mathcal{O}} 	&  \mbox{in}&	Q,		\\
			\noalign{\smallskip}\dis
			y = 0 							& \mbox{on}&	\Sigma,	\\
			\noalign{\smallskip}\dis
			y(\cdot,0) =y_0 							& \mbox{in}&	\Omega
		\end{array}
	\right. 
\end{equation}
and 
\begin{equation}\label{eq:BBM}
	\left |   
		\begin{array}{lcl}
			y_t  -  \Delta y_{t} + \diver(A(x,t)y)   = v\chi_{\mathcal{O}} 	&  \mbox{in}&	Q,  \\
			\noalign{\smallskip}\dis
			y = 0 							& \mbox{on}&	\Sigma, \\
			\noalign{\smallskip}\dis
			y(\cdot, 0) = y_0						& \mbox{in}&	\Omega,
		\end{array}
	\right. 
\end{equation}
where $A=(a_1,\ldots,a_N)$ is a given vector field and $\chi_\mathcal{O}\in C^\infty(Q)$ with $\text{supp}\, \chi_\mathcal{0} \subset \overline{\mathcal{O}}$. 

Our goal in this paper  is to investigate the {\it null controllability problem}:
	\begin{quote}{\it
	given $T>0$ and $y_0\in H^1_0(\Om)$ find a control $v\in L^2(\mathcal{O})$ such that the associated solution of \eqref{eq:BZK} $($resp. \eqref{eq:BBM}$)$ satisfies:
	$$
		y(\cdot,T)=0, \quad \text{in}\quad \Om.
	$$}
	\end{quote}

 
	Equations such as \eqref{generalSobGal} are a particular case of the so called equations of {\it Sobolev-Galpern type}, see \cite{galpern2,sobolev}. 
	These type of equations appear for instance in the study of problems associated with the flow of certain viscous 
	fluids, in the  theory of seepage of homogeneous liquids in fissured rocks, see \cite{bzk}, and surface waves of long
wavelength in liquids, acoustic-gravity waves in compressible fluids, hydromagnetic
waves in cold plasma, acoustic waves in anharmonic crystals, see \cite{bbm}. In particular, equations \eqref{eq:BZK}  and \eqref{eq:BBM} are know as the {\it Barenblatt-Zheltov-Kochina} equation and the multidimensional  Benjamin-Bona-Mahony equation, respectively (see for instance \cite{muitobom,goldstein,bzk,bbm,Milne}).

	Regarding controllability for equations \eqref{eq:BZK} and \eqref{eq:BBM}, as far as we know, the only results available in the literature
	were obtained in the one-dimensional setting. Indeed, in \cite{Micu} it is proved that equation \eqref{eq:BBM}, with $A$ being a constant,
	cannot be steered to zero if $\om \subsetneqq \Om$ is a proper subset. However, the proof given in \cite{Micu} can be only performed in the $1d$ setting, 
	since it relies on the moment method. For a positive controllability result for \eqref{eq:BZK}, we cite \cite{tao_gao_yao}, where the authors consider the problem 
	on the torus and an prove that if one make the control to move in time, in order to cover the whole domain, it is possible to drive the solution to zero. Also related 
	to the controllability of \eqref{eq:BBM}, we cite \cite{yamamoto,zuazua_zhang}, where the unique continuation property is studied.
	
	In this paper, we analyze the null controllability of equations \eqref{eq:BZK} and \eqref{eq:BBM} in the multi-dimensional setting. 
	First, we show that both equations  \eqref{eq:BZK} and \eqref{eq:BBM} cannot be steered to zero if the control is fixed and localized in a proper open subset of $\Omega$. More precisely, 	we prove the following negative results.
\begin{theorem}\label{lacknullBZK}
	Let $T>0$ and $\om \subsetneqq \Om$ be a fixed open set. If $\mathcal{O}=\omega\times (0,T)$ then system \eqref{eq:BZK} is not null controllable at time $T$,
	i.e., there exists $y_0\in H^2(\Omega) \times H^1_0(\Omega)$ such that the null controllability of system \eqref{eq:BZK} fails.
	
\end{theorem}

\begin{theorem}\label{lacknull_BBM}
	Let $T>0$, $\om \subsetneqq \Om$ be a fixed open set and $A\in C^\infty(\overline Q)$. If $\mathcal{O}=\omega\times (0,T)$ then  system \eqref {eq:BBM} is not null controllable at time $T$, i.e., there exists $y_0\in H^2(\Omega) \times H^1_0(\Omega)$ for which the null controllability of \eqref{eq:BBM} does not hold.
\end{theorem}
It is worth to mention that Theorems \ref{lacknullBZK} and \ref{lacknull_BBM} are closely related to the fact that the principal part of \eqref{eq:BZK} and \eqref{eq:BBM}, given by $\partial_t\Delta$, has vertical characteristic hyperplanes which makes impossible to recover any information localized along these characteristics (see Section \ref{sec:neg}). In fact, the proof of both results relies on the construction of highly localized solutions (gaussian beams). For Theorem \ref{lacknullBZK} we construct such solutions by means of  Fourier transform (for similar constructions see \cite{Macia, Debayan}). On the other hand, since the vector field $A$ in equation \eqref{eq:BBM} depends on both the space and time variables,  we can not use Fourier Transform to prove Theorem \ref{lacknull_BBM}. Therefore, we will use a different approach based on  asymptotic expansion of solutions.

	The second main part of this paper is devoted to obtain positive null controllability results for equations \eqref{eq:BZK} and \eqref{eq:BBM}. In fact, since 
	the main obstruction to the null controllability with localized fixed controls is the existence of concentrated solutions, we ask the control to move so that
	we can see the information that would be lost otherwise, i.e., we make the control to move in time in order to cover the whole space domain.

	Before stating the main positive results, we give the precise definition on the movement of the controls.  	
	We take the control domain determined by the evolution of a given reference subset 
	through a flow $X:\mathbb{R}^N\times [0,T]\times [0,T]\to \mathbb{R}^N$, which is generated by some vector field 
	$F \in C([0,T]; W^{2,\infty}(\mathbb{R}^N;\mathbb{R}^N))$, i.e. $X$ solves
\be\label{Xflow}
	\left\{ 
		\begin{array}{l}
		\displaystyle \frac{\partial X}{\partial t} (x , t , t_0)= F ( X ( x,t,t_0),t),\\[3mm]
		X(x, t_0 , t_0)=x. 
		\end{array}
	\right.
\ee
	We make following geometric requirements:
\begin{assumption}\label{AssumptionMC}
	There exist a smooth bounded domain $\omega _0\subset \R ^N$, a curve $\Gamma \in C^\infty ([0,T];\R ^N)$, 
	and two times $t_1,t_2$ with $0\le t_1 < t_2 \le T $ such that: 
\ba
	&& \Gamma (t) \in X(\omega _0,t,0)\cap \Omega , \quad \forall t\in [0,T]; \label{A3a} \\[3mm]
	&&\overline{\Omega} \subset \displaystyle\bigcup\limits_{t\in [0,T]} X(\omega _0,t,0) =\{ X(x,t,0);\ \ (x,t)\in \omega _0\times [0,T]\} ; \label{A3b} \\[3mm]
	&&\Omega \setminus \overline{X(\omega _0,t,0) } \text{ is nonempty and connected for } t\in [0,t_1]\cup [t_2,T]; \label{A3c}\\ [3mm]
	&&\Omega \setminus \overline{X(\omega _0,t,0) } \text{ has two (nonempty) connected components  for } t\in (t_1,t_2); \label{A3d}\\[3mm]
	&&\forall \theta\in C([0,T]; \Omega ) , \ \exists t\in [0,T],\quad \theta (t) \in X(\omega _0,t,0), \label{A3e}
\ea
	where the flow $X$ is  generated by an admisible velocity field $F\in C([0,T]; W^{2,\infty}(\mathbb{R}^N;\mathbb{R}^N))$.
\end{assumption}
  
\begin{definition}[Moving control region]\label{def:O}  
 	A moving control region is defined as $\mathcal{O}_\om=\bigcup_{t\in[0,T]}\left[X(\omega,t,0)\cap \Omega\right]\times\{t\}$ and for any $t>0$ a time section is defined as
	$\mathcal{O}_\om(t)=X(\omega,t,0)\cap \Omega$.
\end{definition}
	We will prove the following positive controllability results.
  
  \begin{theorem}\label{teo_pos_bzk}
	Suppose that Assumption \ref{AssumptionMC} holds and let $T>0$, $\om\subset\mathbb{R}^N$, with $\overline{\om}_0\subset \om$.
	Then, for any $y_0 \in H^2(\Omega) \cap H^1_0(\Omega)$ there exists a moving control $v\in L^2(\mathcal{O}_\om)$ 
	such that the associated solution to \eqref {eq:BZK} satisfies 
 $$
	 y(\cdot,T) =0\quad \hbox{in}\quad \Omega.
 $$
 \end{theorem}
 
 \begin{theorem}\label{teo_pos_bbm}
	Suppose that Assumption \ref{AssumptionMC} holds and let $T>0$, $\om\subset\mathbb{R}^N$, with $\overline{\om}_0\subset \om$.
	Then, for any $y_0 \in H^2(\Omega) \cap H^1_0(\Omega)$ and 
	any $T>0$, there exists a moving control $v\in L^2(\mathcal{O}_\om)$ such that the associated solution to \eqref {eq:BZK} satisfies 
 $$
	 y(\cdot,T) =0\quad \hbox{in}\quad \Omega.
 $$
 \end{theorem}

It is worth mentioning that the idea of making the control to move in order make the system become controllable has been used for many different problems in the past few years, see \cite{CSZZ,Khapalov,KK,MRR}. Here we will use the approach introduced in  \cite{CSZR}, based on Carleman inequalities, which allows us to treat multi-dimensional problems.




\section{Negative Controllability Results}\label{sec:neg}

\subsection{Barenblatt-Zheltov-Kochina with fixed controls} 

We prove Theorem \ref{lacknullBZK}. Here we assume that $\mathcal{O}$ is of the form $\omega\times (0,T)$, where $\omega$ is a proper open subset of $\Omega$. 

For analyzing the controllability of \eqref{eq:BZK} we will make use of the following decomposition:
	\begin{equation}\label{system:BZK}
	\left |   
		\begin{array}{lcl}
			u  - \Delta u  = w			&\mbox{in}&	Q,		\\
			\noalign{\smallskip}\dis
			w_t + w  = u+ v\chi_{\mathcal{O}}	&\mbox{in}&	Q,		\\
			\noalign{\smallskip}\dis
			u = 0						&\mbox{on}&	\Sigma, 	\\
			\noalign{\smallskip}\dis
			w(\cdot,0) = u_0-\Delta u_0					&\mbox{in}&	\Omega.
		\end{array}
	\right. 
\end{equation}
	Indeed,  the solution of  equation \eqref{eq:BZK} satisfies $u(\cdot,T)=0$ if and only if 
	the solution of system \eqref{system:BZK} satisfies $w(\cdot,T)=0$.

	From  duality arguments, the null controllability for system \eqref{system:BZK} with control supported in $\omega\times (0,T)$ is equivalent to the existence of a constant $C>0$ such that the observability inequality 
\[
	\|\psi(\cdot,0)\|^2_{ L^2(\Omega)} \leq C \int^T_0\!\!\!\int_{\omega}|\psi|^2dxdt,
\]
	holds for all $\psi_T \in L^2(\Omega)$, where $\psi$, together with $\varphi$, is the solution of the adjoint system
\begin{equation}\label{adj:BZK}
	\left |   
		\begin{array}{lcl}
			\varphi - \Delta \varphi = \psi 	&\mbox{in}&	Q,  		\\
			\noalign{\smallskip}\dis
			- \psi_t + \psi =\varphi 	&\mbox{in}&	Q, 		\\
			\noalign{\smallskip}\dis
			\varphi= 0     			&\mbox{on}&	\Sigma, 	\\
			\noalign{\smallskip}\dis
			\psi(T) = \psi_T 		&\mbox{in}&	\Omega.
		\end{array}
	\right. 
\end{equation}

	Theorem \ref{lacknullBZK} is a direct consequence of the following proposition:

\begin{proposition}\label{Theo:main}
	Let $\om_0$ be an open subset of $\Om$ such that $\bar\om_0\subsetneq\Om$. Then, there exist $\eps_0>0$ and 
	$\psi_T^\eps\in L^2(\Om)$ such that for any  integer $k>N/4$ the corresponding solution of 
\begin{equation}\label{eq:adjointpseudo_teo}
	\left |   
		\begin{array}{lcl}
			\varphi^\eps  - \Delta \varphi^\eps = \psi^\eps 		&  \mbox{in}	&Q,		\\
			\noalign{\smallskip}\dis
			-\psi^\eps_t  +
			 \psi^\eps=  \varphi^\eps  		&  \mbox{in}	&Q,		\\
			\noalign{\smallskip}\dis
			\varphi^\eps = 0 										& \mbox{on}	&\Sigma,	\\
			\noalign{\smallskip}\dis
			\psi^\eps(\cdot,T)=\psi^\eps_T							&  \mbox{in}	&\Om
		\end{array}
	\right. 
\end{equation}	
	satisfies
\begin{equation}\label{eq:notobs}
	\| \psi^{\epsilon}(\cdot,0)\|^2_{L^2(\Om)} \geq C 
	\quad 
	\hbox{and}
	\quad
	\| \psi^{\epsilon}\|^2_{L^2(0,T;L^2(\om_0))} \leq C\epsilon^{k-N/4}\quad \forall \eps\in(0,\eps_0)
\end{equation}
	where $C$ is a positive constant independent of $\eps$.
\end{proposition}
\begin{proof}

	Let us first consider the system  \eqref{eq:adjointpseudo_teo} posed in $\mathbb{R}^N\times(0,T)$, i.e. 
\begin{equation}\label{XF1}
	\left |   
		\begin{array}{lcl}
			\varphi - \Delta \varphi = \psi 	&  \mbox{in}	&\mathbb{R}^N\times(0,T),  		\\
			\noalign{\smallskip}\dis
			- \psi_t + \psi =\varphi 		&  \mbox{in}	&\mathbb{R}^N\times(0,T),			\\
			\noalign{\smallskip}\dis
			\psi(\cdot,T) = \psi_T 			&  \mbox{in}	&\mathbb{R}^N,
		\end{array}
	\right. 
\end{equation}
with $\psi_T \in  L^2(\mathbb{R}^N)$.

Taking the spatial Fourier transform, one verifies that 
$$
 \hat{\psi}(\xi,t)= e^{-\frac{ |\xi|^2}{ (1+  |\xi|^2)}(T-t)}\hat{\psi}_T(\xi)
 \quad\hbox{and}\quad
 \hat{\varphi}(\xi,t)= {e^{-\frac{ |\xi|^2}{ (1+  |\xi|^2)}(T-t)}\over 1+|\xi|^2}\hat{\psi}_T(\xi) 
$$
solves
\begin{equation}\label{eq:fourier0}
		\left |   
		\begin{array}{lcl}
			(1 +|\xi|^2)\hat{\varphi} = \hat{\psi}	&  \mbox{in}	&\mathbb{R}^N\times(0,T),  		\\
			\noalign{\smallskip}\dis
			- \hat{\psi}_t + \hat{\psi} =\hat{\varphi} 		&  \mbox{in}	&\mathbb{R}^N\times(0,T),			\\
			\noalign{\smallskip}\dis
			\hat\psi(\cdot,T) = \hat\psi_T 			&  \mbox{in}	&\mathbb{R}^N,
		\end{array}
	\right. 
\end{equation}
	where $\hat{\psi}_T$ is the Fourier transform of ${\psi}_T.$

	Now let $\theta$ be a real smooth function supported in $B_1(0)$ with $\|\theta\|_{L^2(\mathbb{R}^N)}=1$ and for each $\epsilon>0$ consider
\begin{equation}\label{Fou1}
 \hat{\psi}_T^{\epsilon}(\xi)= \epsilon^{N/4}\theta\left(\sqrt{\epsilon}\left(\xi - \frac{\overline{\xi}}{\epsilon}\right)\right)e^{-ix_0\cdot \xi},
\end{equation}
where $\overline{\xi} \in \mathbb{R}^N$, $|\overline{\xi}| =1$ and $x_0$ is a point around which we will localize our solution. 

Let $(\hat{\psi}^{\epsilon},\hat{\varphi}^{\epsilon})$ be the solution of \eqref{eq:fourier0} associated to 
$\hat{\psi}_T^{\epsilon}$. 
Since $\hat{\psi}_T^{\epsilon}\in L^2(\mathbb{R}^N)$, let $(\check{\psi}^{\epsilon},\check{\varphi}^{\epsilon})$ be the solution of \eqref{XF1} with initial datum $\check{\psi}_T^{\epsilon}$, the inverse Fourier transform of $\hat{\psi}_T^{\epsilon}$.

\null

\begin{claim}\label{C:C1} There exist two constants $C_1, C_2>0$, independent of $\epsilon$, such that 
$$
	C_1 \leq \|\check{\psi}^{\epsilon} (\cdot,0)\|_{L^2(\mathbb{R}^N)} \leq C_2.
$$
\end{claim}
\begin{proof}[Proof of Claim \ref{C:C1}]

 We have 
\begin{equation}\label{Fou2}
 \check{\psi}^{\epsilon}(x,t) = \frac{1}{(2\pi)^N}\int_{\mathbb{R}^N}e^{-\frac{ |\xi|^2}{ (1+  |\xi|^2)}(T-t)}\hat{\psi}_T^{\epsilon}(\xi) e^{ix\cdot \xi}d\xi
\end{equation}
and by Parseval's identity 
$$
\| \check{\psi}^{\epsilon}(\cdot,t)\|^2_{L^2(\mathbb{R}^N)} =  \frac{1}{(2\pi)^{2N}}\int_{\mathbb{R}^N}e^{-2\frac{ |\xi|^2}{ (1+  |\xi|^2)}(T-t)}|\hat{\psi}_T^{\epsilon}(\xi)|^2d\xi.
$$

	Since $\|\hat{\psi}_T^{\epsilon}\|_{ L^2(\mathbb{R}^N)}=1$, it follows that
\begin{equation}\label{eq:leqgec}
\frac{e^{-2T}}{(2\pi)^{2N}} \leq \| \check{\psi}^{\epsilon}(\cdot,0)\|^2_{L^2(\mathbb{R}^n)}\leq \frac{1}{(2\pi)^{2N}}.
\end{equation}
\end{proof}


\null

\begin{claim}\label{C:C2} Let $x_0\in \mathbb{R}^N$. For any $\delta>0$ there exists $C>0$, independent of $\epsilon$, such that 
$$
\| \check{\varphi}^{\epsilon}\|^2_{L^2(0,T; H^1(|x-x_0| \geq \delta))}+\| \check{\psi}^{\epsilon}\|^2_{L^2(0,T; L^2(|x-x_0| \geq \delta))} \leq C\epsilon^{k-N/4}.
$$
\end{claim}
\begin{proof}[Proof of Claim \ref{C:C2}]

Let us show the estimate for $\check{\varphi}^{\epsilon}$. Similar arguments give
the estimate for  $\check{\psi}^{\epsilon}$.

Since 
\begin{equation}\label{Fou3}
 \check{\varphi}^{\epsilon}(x,t) = \frac{1}{(2\pi)^N}\iint_{\mathbb{R}^N} {e^{-\frac{ |\xi|^2}{ (1+  |\xi|^2)}(T-t)}\over 1+|\xi|^2}\hat{\psi}_T^{\epsilon}(\xi) e^{ix\cdot \xi} d\xi,
\end{equation}
by the change of variables $\zeta = \sqrt{\epsilon}(\xi - \frac{\overline{\xi}}{\epsilon})$ we see that 
\begin{align}
\check{\varphi}^{\epsilon}(x,t) = \frac{\epsilon^{N/4-N/2}}{(2\pi)^N}
\iint_{|\zeta|\leq 1} \theta(\zeta)e^{i(x-x_0)\cdot (\frac{\zeta}{\sqrt{\epsilon}}+ \frac{\overline{\xi}}{\epsilon})}{e^{-\frac{ |\frac{\zeta}{\sqrt{\epsilon}}+ \frac{\overline{\xi}}{\epsilon}|^2}{ (1+  |\frac{\zeta}{\sqrt{\epsilon}}+ \frac{\overline{\xi}}{\epsilon}|^2)}(T-t)}\over 1+|\frac{\zeta}{\sqrt{\epsilon}}+ \frac{\overline{\xi}}{\epsilon}|^2}d\zeta
\end{align}

From the fact that 
%
$$
\Delta^k_{\zeta}e^{i(x-x_0)\cdot (\frac{\zeta}{\sqrt{\epsilon}}+ \frac{\overline{\xi}}{\epsilon})} = (-1)^k\left(\frac{|x-x_0|^2}{\epsilon} \right)^ke^{i(x-x_0)\cdot (\frac{\zeta}{\sqrt{\epsilon}}+ \frac{\overline{\xi}}{\epsilon})}\quad k\in \mathbb{N},
$$
for $|x-x_0| \geq \delta$ and for any  integer $k>N/4$, we have
\begin{align}
\check{\varphi}^{\epsilon}(x,t) 
&= (-1)^k\frac{\epsilon^{k-N/4}}{(2\pi)^N|x-x_0|^{2k}}\iint_{|\zeta|\leq 1}e^{i(x-x_0)\cdot (\frac{\zeta}{\sqrt{\epsilon}}+ \frac{\overline{\xi}}{\epsilon})}   \Delta^k_{\zeta}\bigl({\theta(\zeta)e^{-\frac{ |\frac{\zeta}{\sqrt{\epsilon}}+ \frac{\overline{\xi}}{\epsilon}|^2}{ (1+  |\frac{\zeta}{\sqrt{\epsilon}}+ \frac{\overline{\xi}}{\epsilon}|^2)}(T-t)}\over 1+|\frac{\zeta}{\sqrt{\epsilon}}+ \frac{\overline{\xi}}{\epsilon}|^2}\bigl)d\zeta
\end{align}
	For $\epsilon$ small, one can prove that  the term in $\Delta^k_{\zeta}$ in the above integral is bounded uniformly with respect to $\epsilon$ and then the following estimate holds


\begin{equation}\label{AA1}
|\check{\varphi}^{\epsilon}(x,t) | \leq C\frac{\epsilon^{k-N/4}}{|x-x_0|^{2k}}.
\end{equation}

Analogously, we have 
\begin{equation}\label{AA2}
|\nabla\check{\varphi}^{\epsilon}(x,t) | \leq C\frac{\epsilon^{k-N/4}}{|x-x_0|^{2k}}
\end{equation}
and this gives the estimate for $\check{\varphi}^\epsilon$.
\end{proof}

\begin{claim}\label{l:3}
	Let $\hat\psi_T^\eps$ as \eqref{Fou1} and $(\check{\varphi}^\eps,\check{\psi}^\eps)$ the associated solution 
	of \eqref{XF1}. Then, 
$$
\| \check{\psi}^{\epsilon}(\cdot,0)\|^2_{L^2(|x-x_0| \leq \delta)} \geq C>0.
$$	
\end{claim}
\begin{proof}

From \eqref{AA1} for $t=0$, we get
$$
\| \check{\psi}^{\epsilon}(\cdot,0)\|^2_{L^2(|x-x_0| \geq \delta)} \leq C\epsilon^{k-N/4}
$$
and from Claim \ref{C:C1} we have
$$
\frac{e^{-2T}}{(2\pi)^{2N}} \leq \| \check{\psi}^{\epsilon}(\cdot,0)\|^2_{L^2(\mathbb{R}^n)},
$$
which gives the result.


\end{proof}

We now finish the proof of Proposition \ref{Theo:main}. To do that, consider $x_0\in\Om\backslash\bar\om_0$ and 
$$
0<\eta<\min\{dist(x_0,\partial{\Om}),dist(x_0,\partial{\om_0})\}
$$
	such that
$
	\{x:|x-x_0|\leq \eta\}\subset \Om.
$

As before, take $(\check{\psi}^{\epsilon},\check{\varphi}^{\epsilon})$  the solution of \eqref{XF1} associated to $\check{\psi}_T^{\epsilon}$, the inverse Fourier transform of $\hat{\psi}_T^{\epsilon}$.

	Consider $(\bar\psi^\eps,\bar\varphi^\eps)$ the restriction of $(\check{\psi}^\eps, \check{\varphi}^\eps)$ to $\Om\times(0,T)$. Thus,
\begin{equation}\label{eq:adjointpseudobar}
		\left |   
		\begin{array}{lcl}
			\bar\varphi^\eps  - \Delta \bar\varphi^\eps = \bar\psi^\eps 		&  \mbox{in}	&Q,		\\
			\noalign{\smallskip}\dis
			-\bar\psi^\eps_t  + \bar\psi^\eps=  \bar\varphi^\eps  		&  \mbox{in}	&Q,		\\
			\noalign{\smallskip}\dis
			\bar\varphi^\eps = q^\eps 										& \mbox{on}	&\Sigma,	\\
			\noalign{\smallskip}\dis
			\bar\psi^\eps(\cdot,T)=\psi^\eps_T							&  \mbox{in}	&\Om
		\end{array}
	\right. 
\end{equation}
	where $\bar\psi^\eps_T:= \check{\psi}^\eps_T|_{\Om\times(0,T)}$ and $q^\eps:=\check{\varphi}^\eps\big|_{\partial\Om\times(0,T)}$. 

	From Claim \ref{C:C2}  and Claim \ref{l:3}, we have that
\begin{equation}\label{AA3}
	\| \bar\psi^{\epsilon}\|^2_{L^2(0,T; L^2(\om_0))} \leq C\epsilon^{k-N/4} \quad  \text{and} \quad \|\bar\psi^{\epsilon}(\cdot,0)\|^2_{L^2(\Om)}\geq C>0,
\end{equation}

respectively.



	Now, let  $(\varphi_{\eps}^{\star},\psi_{\eps}^{\star})$ be the solution of
\[
	\left|   
		\begin{array}{lcl}
			\varphi_{\eps}^{\star}  - \Delta \varphi_{\eps}^{\star} = \psi_{\eps}^{\star}  	&  \mbox{in}&Q,		\\
			\noalign{\smallskip}\dis
			-\psi_{\eps,t}^{\star}  + \psi_{\eps}^{\star} = \varphi_{\eps}^{\star} 			&  \mbox{in}&Q,		\\
			\noalign{\smallskip}\dis
			\varphi_{\eps}^{\star} = -q^\eps 										& \mbox{on}&\Sigma,	\\
			\noalign{\smallskip}\dis
			\psi_{\eps}^{\star}(\cdot,T)=0										&\mbox{in}&\Om.
		\end{array}
	\right. 
\]

	Noticing that  $q^\eps\in L^2(0,T;H^{1/2}(\partial\Om))$, one can show that  $\psi_{\eps}^{\star}\in H^1(0,T;L^2(\Om))$ and 
	the following estimate holds
$$
	\|\psi_{\eps}^{\star} \|_{H^1(0,T;L^2(\Om))}\leq C\|q^\eps\|_{L^2(0,T;H^{1/2}(\partial\Om))}.
$$
	
Nevertheless, because $q^\eps:=\check{\varphi}^\eps\big|_{\partial\Om\times(0,T)}$, by trace estimate and Claim \ref{C:C2},
we deduce that
\begin{equation}\label{AA4}
		\|\psi_{\eps}^{\star} \|_{H^1(0,T;L^2(\Om))}\leq C\epsilon^{k-N/4}.
\end{equation}


%

	Finally, defining $(\psi^\eps,\varphi^\eps)=(\bar\psi_{\eps}+\psi_{\eps}^{\star},\bar\varphi_{\eps}+\varphi_{\eps}^{\star})$, we see that $(\psi^\eps,\varphi^\eps)$ solves
	\eqref{eq:adjointpseudo_teo} and by
	\eqref{AA3}--\eqref{AA4} we obtain \eqref{eq:notobs}. 
\end{proof}

\subsection{Benjamin-Bona-Mahony with fixed controls} 

We now prove Theorem \ref{lacknull_BBM}. Here we assume that $\mathcal{O}=\omega\times (0,T)$, where $\omega$ is a proper open subset of $\Omega$.

The null controllability for system \eqref{eq:BBM} with control supported in $\omega\times (0,T)$ is equivalent to the existence of a constant $C>0$ such that the observability inequality 
\begin{equation}\label{obs:BBM}
	\|\psi(\cdot,0)\|^2_{ L^2(\Omega)} \leq C \int^T_0\!\!\!\int_{\omega}|\psi|^2dxdt,
\end{equation}
	holds for all $\psi_T \in L^2(\Omega)$ and $\psi$ is the solution of the adjoint equation
\begin{equation}\label{adj:BZK}
	\left |   
		\begin{array}{lcl}
			-\psi_t  + \Delta \psi_{t} - A\cdot \nabla \psi   =0 	&\mbox{in}&	Q, 		\\
			\noalign{\smallskip}\dis
			\psi= 0     			&\mbox{on}&	\Sigma, 	\\
			\noalign{\smallskip}\dis
			\psi(\cdot,T) = \psi_T 		&\mbox{in}&	\Omega.
		\end{array}
	\right. 
\end{equation}

	In order to prove Theorem \ref{lacknull_BBM}, we show that the observability inequality \eqref{obs:BBM} does not hold for every $\psi_T \in L^2(\Omega)$. 
	
	Given $x_0\in\Om\backslash\overline\om_0$, we set $\alpha(x)=x\cdot \xi_0+i{|x-x_0|^2\over 2}$ with $\xi_0\in \mathbb{R}^N\backslash \{0\}$ 
	and let $\delta>0$ be such that $B_\delta(x_0) \subset \Omega$ and $B_\delta(x_0)\cap \overline{\omega} = \emptyset$. For $h>0$,  we introduce the function
	$$
		\psi_h(x,t)=e^{i{\alpha(x)\over h}}\left(f_0(x)+hf_1(x,t)+h^2f_2(x,t)\right),
	$$
	where 
\begin{equation}\label{adj:BZK}
	\left\{
		\begin{array}{lcl}
			f_0 \in C^{\infty}_0(B_\delta(x_0)), \ \ \ f_0\equiv 1 \ \ \text{in}\ \  (B_{\delta\over2}(x_0), 		\\
			\noalign{\smallskip}\dis
			 f_{1}(x,t) =-if_0(x){\dis\int_t^T \!\!\!A(x,\tau) d\tau\cdot \nabla \alpha(x)\over |\nabla \alpha(x)|^2}, 	\\
			\noalign{\smallskip}\dis
			f_{2}(x,t) ={\dis-\int_t^T \!\!\!A(x,\tau) d\tau\cdot \nabla f_0-i\int_t^T \!\!\!f_{1}(x,\tau)A(x,\tau) d\tau\cdot \nabla \alpha -
			  2i \nabla f_{1} \cdot \nabla \alpha - if_{1} \Delta \alpha \over |\nabla \alpha(x)|^2}.
		\end{array}
	\right. 
\end{equation}	
\begin{remark}\label{rmk:om}
	Since $|\nabla \alpha(x)| \geq |\xi_0|\neq 0$ for all $x \in \overline{\Omega}$, $f_1$ and $f_2$ are well-defined and 
	$\text{supp}\,f_1(\cdot,t)\subset \text{supp}\,f_0$, $\text{supp}\,f_2(\cdot,t) \subset \text{supp}\,f_0$ for all $t\in [0,T]$. 
\end{remark}
	
	It is easy to check that $\psi_h\in C^\infty(\overline{Q})$ satisfies
\begin{equation}\label{sol:psilambda}
	\left |   
		\begin{array}{lcl}
			-\psi_{h,t}  + \Delta \psi_{h,t} - A\cdot \nabla \psi_{h}   = R 	&  \mbox{in}&	Q,  \\
			\noalign{\smallskip}\dis
			\psi_{h} = 0 							& \mbox{on}&	\Sigma, \\
			\noalign{\smallskip}\dis
			\psi_{h}(\cdot,T) = e^{i{\alpha\over h}}f_0						& \mbox{in}&	\Omega,
		\end{array}
	\right. 
\end{equation}
	with 
\begin{equation}
	\begin{alignedat}{2}
		R=&~  e^{i{\alpha\over h}}\biggl[ \bigl( -f_{1,t} + \Delta f_{1,t} -A\cdot \nabla f_1+2i \nabla \alpha \cdot\nabla f_{2,t}+i\Delta \alpha f_{2,t}-i A\cdot \nabla \alpha f_2  \bigl)h\\
		&\hspace{1cm}+\bigl( -f_{2,t} + \Delta f_{2,t} -A\cdot \nabla f_2  \bigl)h^{2}\biggr]\\
		:=&~e^{i{\alpha\over h}}\left(hR_1+h^2R_2\right).
		\end{alignedat}
\end{equation}

	Let now $\varphi\in H^1(0,T;H^2(\Om)\cap H^1_0(\Om))$ be the unique solution of
\begin{equation}\label{A2}
	\left |   
		\begin{array}{lcl}
			-\varphi_t  +  \Delta \varphi_{t} - A\cdot \nabla \varphi   = -R 	&  \mbox{in}&	Q,  \\
			\noalign{\smallskip}\dis
			\varphi = 0 							& \mbox{on}&	\Sigma, \\
			\noalign{\smallskip}\dis
			\varphi(\cdot,0) =  0 					& \mbox{in}&	\Omega.
		\end{array}
	\right. 
\end{equation}
	 The function $\psi=\psi_h+ \varphi$ solves

\begin{equation}\label{A2}
	\left |   
		\begin{array}{lcl}
			-\psi_t  +  \Delta \psi_{t} - A\cdot\nabla \psi   = 0 	&  \mbox{in}&	Q,  \\
			\noalign{\smallskip}\dis
			\psi = 0 							& \mbox{on}&	\Sigma, \\
			\noalign{\smallskip}\dis
			\psi(\cdot,T)= e^{i{\alpha\over h}}f_0		+\varphi(\cdot,T)				& \mbox{in}&	\Omega.
		\end{array}
	\right. 
\end{equation}

For $h$ small enough, we have
\begin{align}
 \| R\|_{L^2(\Omega\times (0,T))}^2 = h \int_0^T\int_{B_R(x_0)}e^{-{|x-x_0|^2\over h}}\left|R_1(x,t)+hR_2(x,t)\right| dx\,dt\sim O(h^{N/2+1}).
 \end{align}
 
 From standard energy estimates, one deduce that
\begin{equation}
\| \varphi\|_{L^2(\omega\times (0,T))}^2 
\leq \|R\|^2_{L^2(\Omega\times(0,T))}=O(h^{N/2+1}),
\end{equation}
	for $h$ small enough.
	
	Now, since $\psi_h\big|_{\om\times (0,T)}=0$, it follows that
\begin{equation}\label{eq:om}
\| \psi\|^2_{L^2(\omega\times (0,T))}\sim O(h^{N/2+1}).
\end{equation}

%

On the other hand, we have
\begin{equation}\label{eq:psi0}
\begin{alignedat}{2}
\| \psi(\cdot,0)\|^2_{L^2(\Omega)}
&= \int_{\Omega}e^{-{|x-x_0|^2\over h}} \left|f_0+hf_1+h^2f_2\right|^2dx \\
&\sim O(h^{N/2}).
\end{alignedat}
\end{equation}

	From \eqref{eq:om} and \eqref{eq:psi0}, it follows that the observability inequality \eqref{obs:BBM} cannot hold for every $\psi_T \in L^2(\Omega)$. This proves Theorem \ref{lacknull_BBM}.

\section{Positive Controllability Results}\label{sec:pos}

	This section is devoted to prove Theorems \ref{teo_pos_bzk} and \ref{teo_pos_bbm}.
	First, let us recall the weight functions needed to apply moving controls and 
	their consequences in terms of Carleman inequalities. 	

	In what follows, we assume that $X$ and $\om_0$ satisfy Assumption \ref{AssumptionMC}, and for each open set $\om\subset \mathbb{R}^N$, with $\overline{\om}_0\subset \om$, 
	we choose $\omega_1$, $\omega_2$ nonempty open sets in  $\mathbb{R}^N$  such that 
	$$ \overline{\omega}_0\subset \omega_1,\,\,  \overline{\omega}_1\subset \omega_2,\,\, \overline{\omega}_2\subset \omega.$$ 

	The following weight function is constructed in  \cite{CSZR}.
\begin{lemma}[\cite{CSZR}]\label{weight}
	There exist a positive number $\tau\in (0,\min\{1,T/2\})$ and a function 
	$\eta \in C^\infty ( \overline{\Omega} \times [0,T])$ such that
\ba
	\nabla \eta(x,t) \ne 0,&&\quad t\in[0,T],\ x\in \Omega \setminus \overline{X(\omega _1,t,0)},  \label{P1}\\
	\eta_t (x,t) \ne 0,  &&  \quad  t\in[0,T],\ x\in \Omega\setminus\overline{ X(\omega _1,t,0)},  \label{P2}\\
	\eta_t (x,t)  >0,  &&  \quad  t\in[0,\tau],\ x\in \Omega\setminus \overline{X(\omega _1,t,0)}, \label{P3} \\
	\eta_t  (x,t) <0,  &&  \quad  t\in[T-\tau ,T ],\ x\in \Omega\setminus  \overline{X(\omega _1,t,0)},\label{P4} \\
	\frac{\partial \eta}{\partial \nu}(x,t) \le 0,  && \quad  t\in [0,T ],\ x\in \partial \Omega , \label{P5}\\
	\eta (x,t) >\frac{3}{4}\|\eta \|_{\infty }, &&\quad  t\in [0,T ],\ x\in \overline{ \Omega} .\label{P6}
\ea
\end{lemma}


 Next, we  introduce a real function $r\in C^\infty(0,T)$, symmetric with respect to $t={T\over2}$ and such that for $\tau>0$, as above, 
\[
	r(t) = \left\{ 
	\begin{array}{lcl}
	\noalign{\smallskip}\dis
	{1\over t} &\text{for}& 0<t\leq{\tau\over2}, \\
	\noalign{\smallskip}\dis
	\text{\rm strictly decreasing}&\text{for}& {\tau\over2}< t <\tau,\\
	\noalign{\smallskip}\dis
	1	&\text{for}&\tau \le t \le {T\over2},\\
	\noalign{\smallskip}\dis
	r(T-t)	&\text{for}& {T\over2} \le t < T
	\end{array}
	\right.
\]
 and define the weights
\[
	\begin{array}{lr}
	\gamma(x,t)= e^{\lambda \eta(x,t)}&(x,t)\in \Omega  \times (0,T),\\
	\noalign{\smallskip}\dis
	\alpha (x,t) = r(t) ( e^{2\lambda\|\eta \|_\infty }- \gamma(x,t))&(x,t)\in \Omega \times (0,T),\\
	\noalign{\smallskip}\dis
	\xi (x,t) = r(t) \gamma(x,t)&(x,t)\in \Omega  \times (0,T),\\
	\alpha^* (t) = \max\limits_{x\in \overline\Omega} \alpha(x,t) \quad &t\in  (0,T),\\
		\noalign{\smallskip}\dis
	\xi^* (t) =\min\limits_{x\in \overline\Omega} \xi(x,t)\quad &t\in  (0,T).
\end{array}
\] 
	where $\lambda >0$ is a parameter that will be chosen large enough.

	The following Carleman inequality is proved in \cite{CSZR}.
\begin{lemma}
\label{lemma:ode}
	There exist positive real numbers $\lambda _1> 0$,  $s_1> 0$ and $C_1>0$ $($depending on $\Om$ and $\om_0$$)$
	 such that for all $\lambda \ge \lambda _1$, all $s\ge s_1$ and all $q\in H^1(0,T;L^2(\Omega))$, 
	 the following inequality holds
\[
	s\lambda^2\iint_{Q}\xi|q|^2e^{-2s\alpha}\,dxdt\le C_1\left(\iint_{Q}|q_t|^2e^{-2s\alpha}\, dxdt 
	+ s^2\lambda ^2\int_0^T\!\!\!\int_{\mathcal{O}_{\om_2}(t)}e^{-2s\alpha}\xi^2|q|^2\,dxdt\right). 
\]  
\end{lemma}
	\noindent We recall that $\mathcal{O}_{\om_2}(t)=X(\omega_2,t,0)\cap \Om$ (see Definition \ref{def:O}). 

	For our  purposes, we prove the following two new Carleman inequalities for the Laplace operator.
\begin{lemma}
\label{lemma:elliptic}
	There exist positive real numbers $\lambda _2>0$,  $\tau_2>0$ and $C_2>0$, independents of $t$, 
	such that for all $\lambda \ge \lambda _2$, all $\tau\ge \tau_2$ and all 
	$z\in C^0([0,T];H^2(\Omega )\cap H^1_0(\Omega ))$, the following inequality holds
\[
	\ii\left[\lambda ^4(\tau\gamma)^3|z|^2+\lambda^2(\tau\gamma)|\nabla z|^2\right]e^{2\tau\gamma}\,dx
	\le C_2 \left(\!\ii|\Delta z|^2 e^{2\tau \gamma}\,dx
	+\int_{\mathcal{O}_{\om_2}(t)}\!\!\!\lambda ^4(\tau\gamma)^3|z|^2e^{2\tau\gamma}\,dx\!\right),~ \forall t\in [0,T].
\]
\end{lemma} 
\begin{lemma}\label{carleman:elliptic_H-1}
	There exist positive real numbers $\lambda_3>0$, $\tau_3>0$ and $C_3>0$, independents of $t$, 
	such that for all $\lambda \geq \lambda_3$, all $\tau \geq \tau_3$ and all $(g,G)\in  H^1(0,T;L^2(\Om)\times L^2(\Om)^{N})$, the solution $z$ of 
\begin{equation}\label{transcar}
\left |   
\begin{array}{lcl}
 	- \Delta z  = g+ \nabla \cdot G 	&\mbox{in}& Q,  \\
	z = 0     					&\mbox{on}&\Sigma, \\
\end{array}
\right. 
\end{equation}
	satisfies 
\[
\begin{alignedat}{2}
	\ii\!\!\left[\lambda^2(\tau\gamma)^2 |z|^2\!+\! |\nabla z|^2\!\right]e^{2\tau\gamma}dx 
	\leq&~C_3 \left(\ii\!\!\left[\lambda^{-2}(\tau\gamma)^{-1}|g|^2
	\!+\!(\tau\gamma)|G|^2\right]\! e^{2\tau\gamma} dx+ \!\!\int_{\mathcal{O}_{\om_2}(t)}\!\!\!\!\!\!\!\lambda^2(\tau\gamma)^2 |z|^2e^{2\tau\gamma}\, dx\right ),
\end{alignedat}
\]
	for all $ t\in [0,T]$.
\end{lemma}
	\noindent  For sake of completeness, we give a sketch of the  proofs of Lemmas \ref{lemma:elliptic} and \ref{carleman:elliptic_H-1} in Appendixes \ref{App:mov_lapla} and \ref{App:H-1}, respectively.

%
%
%


%
%

 \subsection{Barenblatt-Zheltov-Kochina with moving controls}\label{secBZK}

 In this section we show the null controllability  for equation \eqref{eq:BZK}. In fact, Theorem \ref{teo_pos_bzk} is a direct consequence of the following result.
\begin{proposition}\label{nullBZK_sys}
	Suppose that Assumption \ref{AssumptionMC} holds and let $T>0$, $\om\subset\mathbb{R}^N$, with $\overline{\om}_0\subset \om$. For any $z_0 \in L^2(\Omega)$, there exists a moving control $v \in L^2(\mathcal{O}_{\om})$ such that the solution  
\begin{equation}\label{system:BZK1}
	\left |   
		\begin{array}{lcl}
			y  - \Delta y  = z			&\mbox{in}&	Q,		\\
			\noalign{\smallskip}\dis
			z_t + z  = y+ v \chi_{\mathcal{O}_{\om}}	&\mbox{in}&	Q,		\\
			\noalign{\smallskip}\dis
			y = 0						&\mbox{on}&	\Sigma, 	\\
			\noalign{\smallskip}\dis
			z(\cdot,0) = z_0					&\mbox{in}&	\Omega.
		\end{array}
	\right. 
\end{equation}
satisfies 
$$
	y(\cdot,T)=z(\cdot,T)=0\quad \hbox{in}\quad \Omega.
$$
\end{proposition}




	We prove Proposition \ref{nullBZK_sys} by showing that there exists $C>0$ such that
\begin{equation}\label{adj:BZK_obs}
 	\|\psi(\cdot,0)\|^2_{ L^2(\Omega)} \leq C \int^T_0\!\!\!\int_{\mathcal{O}_\om(t)}|\psi|^2dxdt, \quad \forall\psi_T \in L^2(\Omega),
\end{equation}
	with $(\varphi,\psi)$ solution of	
\begin{equation}\label{adj:BZK}
	\left |   
		\begin{array}{lcl}
			\varphi - \Delta \varphi = \psi 	&\mbox{in}&	Q,  		\\
			\noalign{\smallskip}\dis
			- \psi_t + \psi =\varphi 	&\mbox{in}&	Q, 		\\
			\noalign{\smallskip}\dis
			\varphi= 0     			&\mbox{on}&	\Sigma, 	\\
			\noalign{\smallskip}\dis
			\psi(T) = \psi_T 		&\mbox{in}&	\Omega.
		\end{array}
	\right. 
\end{equation}

	By a straightforward argument, the proof of the observability inequality \eqref{adj:BZK_obs} is reduced to the following 
	Carleman inequality:
\begin{theorem}\label{thm:Carl_BZK}
	Under Assumption \ref{AssumptionMC}. For any $T>0$, $\om\subset\mathbb{R}^N$, with $\overline{\om}_0\subset \om$, there exist positive constants $s_0$, $\lambda_0\geq 1$ and~$C$, only depending on $\Omega$ and $\omega$, such that, 
	for any $\psi _T\in L^2(\Omega)$, the solution $(\varphi,\psi)$ to the adjoint system \eqref{adj:BZK} satisfies:
\begin{align}
	&\iint_Q [s\lambda ^2 \xi  |\nabla \varphi|^2    +  s^3\lambda ^4\xi ^3 |\varphi|^2  ]e^{-2s\varphi} dxdt + 
	\iint_Q s\lambda ^2  \xi |  \psi |^2 e^{-2s\alpha }dxdt \nonumber \\
	&+ \iint_Q [s\lambda ^2 \xi^*  |\nabla \varphi_t|^2    +  s\lambda ^2\xi^* |\varphi_t|^2  ]e^{-2s\alpha^*} dxdt  \nonumber\\
	& \le C s^5\lambda ^6\int_0^T \!\!\! \int_{\mathcal{O}_\om(t)}    \xi ^5 e^{-4s\alpha + 2s\alpha^*}  |\psi|^2  dxdt.  \label{E6T0} 
\end{align}
	for all  $s \geq  s_0(T + T^2) $ and for all $\lambda \geq \lambda_0$.
\end{theorem}
\begin{proof}
	First, applying Lemma \ref{lemma:ode} to $\eqref{adj:BZK}_2$ and taking $\lambda$ large enough, we obtain 
\be  \label{est:ode}
	s \lambda ^2 \iint_Q \xi  |  \psi |^2 e^{-2s\alpha}dxdt \le C_1
	\left(\iint_Q |\varphi|^2 e^{-2s\alpha} dxdt 
	+s^2 \lambda ^2 \int_0^T\!\!\!\int_{\mathcal{O}_{\om_2}(t)}  \xi ^2  |\psi|^2e^{-2s\alpha} dxdt\right),  
\ee  
	
	Next, applying Lemma \ref{lemma:elliptic} to $\eqref{adj:BZK}_1$, we obtain
\[
	\tau^3\lambda ^4\ii\!\!\gamma^3e^{2\tau\gamma}|\varphi|^2\,dx
	+\tau\lambda^2\ii\!\!\gamma e^{2\tau\gamma}|\nabla \varphi|^2\,dx
	\le C_0 \left(\ii\!\! e^{2\tau \gamma}|\psi|^2\,dx
	+\tau^3\lambda ^4\int_{\mathcal{O}_{\om_2}(t)}\!\!\!\!\!\!
	\gamma^3e^{2\tau\gamma}|\varphi|^2\,dx\right), 
\]	
	for $\lambda\geq\lambda_0$ and $\tau\geq \tau_0$.
	
	To connect this elliptic estimate with \eqref{est:ode}, we set
$$
	\tau=sg(t),
$$
	we multiply by
$$
	e^{-2sg(t)e^{2\|\eta\|_\infty}}
$$
	and we integrate with respect to $t$ in (0,T). Let us remark that the last choice of $\tau$ will be greater that $\tau_0$
	whatever we take $s_0\geq \tau_0$, then we deduce
\begin{equation}\label{est:elliptic}
\begin{alignedat}{2}
		&\iint_Q [s\lambda ^2 \xi  |\nabla \varphi|^2    +  s^3\lambda ^4 \xi ^3 |\varphi|^2  ]e^{-2s\alpha} dxdt \\
		&\le C_0 \left( \iint_Q |\psi |^2 e^{-2s\alpha}dxdt 
		+ s^3\lambda ^4 \int_0^T \!\!\! \int_{\mathcal{O}_{\om_2}(t) }  \xi ^3 |\varphi|^2 e^{-2s\alpha} dxdt \right). 
	\end{alignedat}
\end{equation}

Adding \eqref{est:ode} and  \eqref{est:elliptic}, and absorbing the lower order terms by taking $\lambda$ large enough,  we get  
\begin{equation}\label{est:elliptic+ode}
\begin{alignedat}{2}
		&\iint_Q[s\lambda ^2 \xi  |\nabla \varphi|^2    +  s^3\lambda ^4\xi ^3 |\varphi|^2  ]e^{-2s\alpha} dxdt 
		+ \iint_Q s\lambda ^2  \xi |  \psi |^2 e^{-2s\alpha}dxdt \\
		&\le C \left( s^2\lambda ^2\int_0^T \!\!\! \int_{\mathcal{O}_{\om_2}(t)} \xi ^3 |\psi|^2 e^{-2s\alpha} dxdt 
		+ s^3 \lambda ^4\int_0^T \!\!\! \int_{\mathcal{O}_{\om_2}(t)}   \xi ^3 |\varphi|^2 e^{-2s\alpha} dxdt \right).
	\end{alignedat}
\end{equation}

	Next, we need to eliminate the local integral of $\varphi$ appearing in the right hand side of \eqref{est:elliptic+ode}. 
	For that, we first take the time derivative in first equation of system \eqref{adj:BZK} and use the second 
	equation of the same system to see that $\varphi_t$ solves the following  elliptic equation 

\begin{equation}\label{adj:BZKvarphi}
	\left |   
		\begin{array}{lcl}
			\varphi_t - \Delta \varphi_t = \psi -\varphi &     \mbox{in}& Q,  \\
			\noalign{\smallskip}\dis
			\varphi_t= 0     &    \mbox{on} &   \Sigma.
		\end{array}
	\right. 
\end{equation}

Then, from  \eqref{est:elliptic+ode}  and energy estimates  for \eqref{adj:BZKvarphi}, it is not difficult to see that 
\begin{equation}\label{eq:BZKenergy} 
	\begin{alignedat}{2}
	&\int_Q [s\lambda ^2 \xi^*|\nabla \varphi_t|^2 + s\lambda ^2\xi^* |\varphi_t|^2  ]e^{-2s\alpha^*} dxdt \\
 	&\le C \left( s^2\lambda ^2\int_0^T \!\!\! \int_{\mathcal{O}_{\om_2}(t)}\xi ^3 |\psi|^2 e^{-2s\alpha} dxdt 
	+ s^3 \lambda ^4\int_0^T \!\!\! \int_{\mathcal{O}_{\om_2}(t)}   \xi ^3 |\varphi|^2 e^{-2s\alpha} dxdt \right).  
	\end{alignedat}
\end{equation}

	Combining  \eqref{est:elliptic+ode} and  \eqref{eq:BZKenergy}, we get 
\begin{equation}\label{estnew:elliptic+ode}
\begin{alignedat}{2}
		&\iint_Q[s\lambda ^2 \xi  |\nabla \varphi|^2    +  s^3\lambda ^4\xi ^3 |\varphi|^2  ]e^{-2s\alpha} dxdt 
		+ \iint_Q s\lambda ^2  \xi |  \psi |^2 e^{-2s\alpha}dxdt \\
		&+\int_Q [s\lambda ^2 \xi^*|\nabla \varphi_t|^2 + s\lambda ^2\xi^* |\varphi_t|^2  ]e^{-2s\alpha^*} dxdt\\
		\le &~C \left( s^2\lambda ^2\int_0^T \!\!\! \int_{\mathcal{O}_{\om_2}(t)}\xi ^3 |\psi|^2 e^{-2s\alpha} dxdt 
		+ s^3 \lambda ^4\int_0^T \!\!\! \int_{\mathcal{O}_{\om_2}(t)}  \xi ^3 |\varphi|^2 e^{-2s\alpha} dxdt \right).
	\end{alignedat}
\end{equation}

	Now, let us introduce the function 
\begin{equation}\label{zeta}
	\zeta (x,t) := \vartheta(X(x,0,t) ),
\end{equation}
	where $\vartheta$ is a cut-off function satisfying 
\[
	 	\vartheta\in C_0^\infty (\omega  ),\quad
		0\le \vartheta(x)\le 1,\quad
		\vartheta \equiv1\quad\hbox{in}\quad\omega  _2. 
\]

	This way, we have that
\begin{equation}\label{eq:cut-offBZK}
	s^3\lambda ^4 \int_0^T \!\!\! \int_{\mathcal{O}_{\om_2}(t)}   \xi ^3 |\varphi|^2 e^{-2s\alpha} dxdt  
	\leq s^3\lambda ^4 \iint_Q \zeta   \xi ^3 |\varphi|^2 e^{-2s\alpha}  dxdt.
\end{equation}
	Then, using $ \eqref{adj:BZK}_2$, we obtain
\begin{eqnarray*}
	s^3\lambda ^4 \iint_Q \zeta \xi^3 |\varphi|^2 e^{-2s\alpha} dxdt 
	&=&  s^3\lambda ^4 \iint_Q \zeta\xi ^3\varphi\psi e^{-2s\alpha}  dxdt
	+s^3\lambda ^4 \iint_Q \zeta  \xi ^3  \varphi(-\psi_t) e^{-2s\alpha} dxdt   \\
	\noalign{\smallskip}\dis
	&:=& M_1 + M_2.
\end{eqnarray*}

	Now, let us estimate the terms $M_1$ and $M_2$. From Cauchy-Schwarz inequality, we have
\begin{equation}\label{WXYZ1}
	|M_1| \le \delta  s^3\lambda ^4  \iint_Q \xi ^3  |\varphi|^2 e^{-2s\alpha} dxdt  
	+C_{\delta}s^3\lambda ^4 \int_0^T \!\!\! \int_{\mathcal{O}_{\om}(t)} \xi^3  |\psi|^2e^{-2s\alpha}  dxdt, 
\end{equation}
	for any $\delta >0$. 

	On the other hand, integrating by parts with respect to  $t$ in $M_2$, it yields
\begin{eqnarray*}
	M_2 &=&  s^3\lambda ^4 \iint_Q\zeta\xi^3  \varphi_t\psi e^{-2s\alpha}  dxdt   
	+s^3\lambda ^4 \iint_Q\zeta(3\xi^2 \xi_t - 2 s\alpha_t\xi ^3) \varphi\psi e^{-2s\alpha}  dxdt  \\
	&&\qquad +  s^3\lambda ^4\iint_Q\nabla\vartheta(X(x,0,t))\cdot 
  	 {\partial X\over\partial t}(x,0,t) \xi ^3\varphi\psi e^{-2s\alpha}  dxdt\\
   	&:=&M_2^1+M_2^2-M_2^3. 
\end{eqnarray*}
	
	For $M_2^1$, we notice that, for every $\delta >0$, we obtain
\begin{equation}\label{est:A21}
	\begin{alignedat}{2}
		|M_2^1| 
		\le&~ \delta s\lambda^2 \int_Q\xi^*  |\varphi_t|^2 e^{-2s\alpha^*} dxdt
		+C_{\delta}s^6\lambda^6 \int_0^T \!\!\! \int_{\mathcal{O}_{\om}(t)} \xi^6  |\psi|^2
		e^{-4s\alpha + 2s\alpha^*}  dxdt.
	\end{alignedat}
\end{equation}
	
	Since $|\xi_t| +|\alpha_t| \leq s\xi  ^2$, for every $\delta >0$, we infer that
\begin{equation}\label{est:A22}
	\begin{alignedat}{2}
	|M_2^2| 
	\le&~\delta s^3\lambda^4\iint_Q\xi ^{3}|\varphi|^2 e^{-2s\alpha} dxdt
	+C_\delta s^7\lambda^4\int_0^T \!\!\! \int_{\mathcal{O}_{\om}(t)} \xi ^7  |\psi|^2e^{-2s\alpha}  dxdt .
	\end{alignedat}
\end{equation}
	
	Finally, $M_2^3$ is estimated as the term $M_1$:
\begin{equation}\label{est:A23}
	|M_2^3| \leq \delta s^3\lambda^4  \iint_Q \xi ^3 |\varphi|^2  e^{-2s\alpha} dxdt 
	+ C_{\delta} s^3\lambda^4 \int_0^T \!\!\! \int_{\mathcal{O}_{\om}(t)} \xi ^3 |\psi|^2 e^{-2s\alpha} dxdt.
\end{equation}
	for any $\delta >0$. 

Combining \eqref{estnew:elliptic+ode} and  \eqref{eq:cut-offBZK}-\eqref{est:A22}, and absorbing the lower order terms,  we conclude the  proof of Theorem \ref{thm:Carl_BZK}.

\end{proof}

\begin{remark}
	Notice that if the initial data $\psi_T \in H^1_0(\Omega)$, then $\psi$ and $\varphi$ satisfy 
\begin{equation}\label{pseudoZ}
	\left |   
		\begin{array}{lcl}
			-\varphi_t  + \Delta \varphi_t  - \Delta \varphi   = 0	&  \mbox{in}	&Q,		\\
			\noalign{\smallskip}\dis
			\varphi = 0 						& \mbox{on}	&\Sigma,	\\
			\noalign{\smallskip}\dis
			\varphi(\cdot,T)=\varphi_T:= (I-\Delta)^{-1}\psi_T			&  \mbox{in}	&\Omega	
		\end{array}
	\right. 
\end{equation}
	and 
\begin{equation}\label{pseudoXi}
	\left |   
		\begin{array}{lcl}
			-\psi_t  + \Delta \psi_t  - \Delta \psi  = 0		&  \mbox{in}&	Q,		\\
			\noalign{\smallskip}\dis
			\psi = 0 							& \mbox{on}&	\Sigma,	\\
			\noalign{\smallskip}\dis
			\psi(\cdot,T)= \psi_T						&  \mbox{in}&	\Omega,	
		\end{array}
	\right. 
\end{equation}
	respectively. In spite of seeming that the equations are decoupled, one can observe that, their are coupled by the initial data once
	they satisfy an elliptic equation, i.e. $\varphi_T-\Delta \varphi_T=\psi_T$. 
\end{remark}

\begin{remark}
	If $ u_0 \in H^1_0(\Om)$, there is another decomposition of \eqref {eq:BZK}, namely
\begin{equation}\label{systempseudoF}
\left |   
\begin{array}{ll}
u_t  = v &     \mbox{in}  \  \ Q,  \\
 v - \Delta v - \Delta u  = f 1_{\omega(t)}  &     \mbox{in}  \  \ Q, \\
u = 0     &    \mbox{on}  \  \   \Sigma, \\
u(0) = u_0 &    \mbox{in}    \  \  \Omega,
\end{array}
\right. 
\end{equation}
which leads to the adjoint system
\begin{equation}\label{systempseudo1F}
\left |   
\begin{array}{ll}
-\xi_t + \xi = z    &     \mbox{in}  \  \ Q,  \\
z - \Delta z    = \xi &     \mbox{in}  \  \ Q, \\
z= 0     &    \mbox{on}  \  \   \Sigma, \\
\xi(T) = \xi_T &    \mbox{in}    \  \  \Omega.
\end{array}
\right. 
\end{equation}
However, here $\xi_T \in H^{-1}(\Omega)$, for which we do not know how to prove a Carleman inequality for the ODE part of the decomposition.
\end{remark}



 \subsection{Benjamin-Bona-Mahony with moving controls}\label{secBBM}
	In this section we prove the null controllability for the Benjamin-Bona-Mahony equation \eqref{eq:BBM}. 
	Here, we use a slightly  different proof as the one given in the previous section for the
	Barenblatt-Zheltov-Kochina equation.

	Theorem \ref{teo_pos_bbm} is a consequence of the following result:
\begin{proposition}\label{nullBBM_sys}
Suppose that Assumption \ref{AssumptionMC} holds and let $T>0$, $\om\subset\mathbb{R}^N$, with $\overline{\om}_0\subset \om$. Then, for any $z_0 \in L^2(\Omega)$, there exists a moving control $v \in L^2(\mathcal{O}_\om)$ 
	such that the solution  
\begin{equation}\label{system:BBM}
	\left |   
		\begin{array}{lcl}
			y  - \Delta y  = z			&\mbox{in}&	Q,		\\
			\noalign{\smallskip}\dis
			z_t + \nabla\cdot(A(x,t)y)  =  v \chi_{\mathcal{O}_\om}	&\mbox{in}&	Q,		\\
			\noalign{\smallskip}\dis
			y = 0						&\mbox{on}&	\Sigma, 	\\
			\noalign{\smallskip}\dis
			z(\cdot,0) = z_0					&\mbox{in}&	\Omega.
		\end{array}
	\right. 
\end{equation}
satisfies 
$$
	y(\cdot,T)=z(\cdot,T)=0\quad \hbox{in}\quad \Omega.
$$
\end{proposition}

	Once more, one can see that Proposition \ref{nullBBM_sys} is equivalent to find $C>0$ such that 
	\begin{equation}\label{eq:obs_bbm}
	\|\psi(\cdot,0)\|^2_{ L^2(\Omega)} \leq C \int^T_0\!\!\!\int_{\mathcal{O}_\om(t)}|\psi|^2dxdt,\quad \forall\psi_T \in L^2(\Omega),
\end{equation}
	where $(\varphi,\psi)$ solves
\begin{equation}\label{adj:BBM}
	\left |   
		\begin{array}{lcl}
			\varphi - \Delta \varphi = A\cdot\nabla\psi 	&\mbox{in}&	Q,  		\\
			\noalign{\smallskip}\dis
			- \psi_t  =\varphi 	&\mbox{in}&	Q, 		\\
			\noalign{\smallskip}\dis
			\varphi= 0     			&\mbox{on}&	\Sigma, 	\\
			\noalign{\smallskip}\dis
			\psi(T) = \psi_T 		&\mbox{in}&	\Omega.
		\end{array}
	\right. 
\end{equation}

	Inequality \eqref{eq:obs_bbm} is a consequence of the following Carleman inequality:
\begin{theorem}\label{thm:Carl_BBM}
	Under Assumption \ref{AssumptionMC}. For any $T>0$, $\om\subset\mathbb{R}^N$, with $\overline{\om}_0\subset \om$, there exist positive constants $s_0$, $\lambda_0\geq 1$ and~$C$, only depending on $\Omega$ and $\omega$, such that, 
	for any $\psi _T\in L^2(\Omega)$, the solution $(\varphi,\psi)$ to the adjoint system \eqref{adj:BBM} satisfies:
\begin{align*}
	&\iint_Qe^{-2s\alpha} [ |\nabla \varphi|^2    +  s^2\lambda ^2\xi ^2 |\varphi|^2 ]dxdt 
		+  s\lambda ^2  \iint_Q\xi |  \psi |^2 e^{-2s\alpha}dxdt \\
		&+\int_Q [s\lambda ^2 \xi^*|\nabla \varphi_t|^2 + s\lambda ^2\xi^* |\varphi_t|^2  ]e^{-2s\alpha^*} dxdt\\
	& \le C s^6\lambda ^2\int_0^T \!\!\! \int_{\mathcal{O}_\om(t)}    \xi ^6 e^{-4s\alpha + 2s\alpha^*}  |\psi|^2  dxdt.
\end{align*}
	for all  $s \geq  s_0(T + T^2) $ and for all $\lambda \geq \lambda_0$.
\end{theorem}
\begin{proof}
	We begin applying the Carleman inequality given by Lemma \ref{lemma:ode} to $\eqref{adj:BBM}_2$, which gives 
\begin{equation}\label{est:ode_BBM}
	\iint_Q s\lambda ^2  \xi |  \psi |^2 e^{-2s\alpha }dxdt \leq \iint_Q |\varphi|^2e^{-2s\alpha }dxdt 
	+ \int_0^T\!\!\!\!\int_{\mathcal{O}_{\om_2}(t)} s^2\lambda ^2  \xi^2 |  \psi |^2 e^{-2s\alpha }dxdt.
\end{equation}

	Next, noticing that 
\begin{equation*}
	A\cdot\nabla\psi=\nabla\cdot(A\psi)-\psi\nabla\cdot A
\end{equation*}

	and applying the Carleman inequality given in Lemma \ref{carleman:elliptic_H-1} for $\eqref{adj:BBM}_1$,
\[
\begin{alignedat}{2}
	\tau^2\lambda^2 \ii e^{2\tau\gamma}\gamma^2 |\varphi|^2\, dx+\ii e^{2\tau\gamma}|\nabla\varphi|^2\, dx \leq&~ 
	C \left( {1\over\tau\lambda^2}\ii e^{2\tau\gamma}{|\psi|^2\over\gamma} dx
	+ \tau \ii e^{2\tau\gamma}\gamma|\psi|^2 dx\right.\\
	&\left.+\,\tau^2\lambda^2 \int_{\mathcal{O}_{\om_2}(t)} \!\!\!\!\!\!e^{2\tau\gamma}\gamma^2 |\varphi|^2\, dx
	+\int_{\mathcal{O}_{\om_2}(t)} \!\!\!\!\!\!e^{2\tau\gamma}|\nabla\varphi|^2\, dx  \right ),
\end{alignedat}
\]
	for $\lambda\geq\lambda_0$ and $\tau\geq \tau_0$.
	
	As before, to connect this elliptic estimate with \eqref{est:ode_BBM}, we set
$$
	\tau=sg(t),
$$
 multiply by
$$
	e^{-2sg(t)e^{2\|\eta\|_\infty}}
$$
	and  integrate with respect to $t$ in $(0,T)$. The last choice of $\tau$ will be greater 
	that $\tau_0$ whatever we take $s_0\geq \tau_0$, hence we deduce that 
\begin{equation}\label{est:elliptic_BBM0}
\begin{alignedat}{2}
	&s^2\lambda^2 \iint_Q e^{-2s\alpha}\xi^2 |\varphi|^2\, dxdt+\iint_Q e^{-2s\alpha}|\nabla\varphi|^2\, dxdt\\
	& \leq
	C \left( {1\over s\lambda^2}\iint_Q e^{-2s\alpha}{|\psi|^2\over\xi}\, dxdt
	+ s\iint_Q e^{-2s\alpha}\xi|\psi|^2\, dxdt\right.\\
	&\left.\hspace{0.5cm}
	+\,s^2\lambda^2 \int_0^T\!\!\!\!\int_{\mathcal{O}_{\om_2}(t)} \!\!\!\!\!\!e^{-2s\alpha}\xi^2 |\varphi|^2\, dxdt
	+\int_0^T\!\!\!\!\int_{\mathcal{O}_{\om_2}(t)} \!\!\!\!\!\!e^{-2s\alpha}|\nabla\varphi|^2\, dxdt  \right ).
\end{alignedat}
\end{equation}

	Adding \eqref{est:ode_BBM} and  \eqref{est:elliptic_BBM0}, and absorbing the lower order terms by taking $\lambda$ large enough,  we get
\begin{equation*}
\begin{alignedat}{2}
	& s\lambda ^2\iint_Q   e^{-2s\alpha }\xi |  \psi |^2dxdt+s^2\lambda^2 \iint_Q e^{-2s\alpha}\xi^2 |\varphi|^2\, dxdt+\iint_Q e^{-2s\alpha}|\nabla\varphi|^2\, dxdt\\
	& \leq
	C \left(  s^2\lambda ^2\int_0^T\!\!\!\!\int_{\mathcal{O}_{\om_2}(t)}  e^{-2s\alpha }\xi^2 |  \psi |^2 dxdt
	+\,s^2\lambda^2 \int_0^T\!\!\!\!\int_{\mathcal{O}_{\om_2}(t)}\!\!\!\!\!\!e^{-2s\alpha}\xi^2 |\varphi|^2\, dxdt\right.\\
	&\left.\hspace{0.5cm}
	+\int_0^T\!\!\!\!\int_{\mathcal{O}_{\om_2}(t)}\!\!\!\!\!\!e^{-2s\alpha}|\nabla\varphi|^2\, dxdt  \right ).
\end{alignedat}
\end{equation*}

Now, let us introduce $\om_3$ such that $\overline\om_2\subset\om_3\subset\overline\om_3\subset\om$ and the function 
\begin{equation*}
	\zeta (x,t) := \vartheta(X(x,0,t) ),
\end{equation*}
	where $\vartheta$ is a cut-off function satisfying 
\[
	\vartheta\in C_0^\infty (\omega_3  ), \quad 	0\le \vartheta(x)\le 1, \quad	\vartheta \equiv1\quad\hbox{in}\quad\omega_2. \label{AW3}
\]

This way, we have that
\begin{equation}\label{eq:cut-offBBM}
	\int_0^T\!\!\!\!\int_{\mathcal{O}_{\om_2}(t)} \!\!\!\!\!\!e^{-2s\alpha}|\nabla\varphi|^2\, dxdt\le 
	\iint_Q \zeta e^{-2s\alpha}|\nabla\varphi|^2\, dxdt.
\end{equation}
	Then, since $\zeta e^{-2s\alpha}\nabla\varphi=\nabla(\zeta e^{-2s\alpha}\varphi)-\nabla(\zeta e^{-2s\alpha})\varphi$,
	we obtain
\begin{eqnarray*}
	\iint_Q \zeta e^{-2s\alpha}|\nabla\varphi|^2\, dxdt
	&=&  \iint_Q \nabla(\zeta e^{-2s\alpha}\varphi)\cdot\nabla\varphi\, dxdt
	-\iint_Q [\nabla(\zeta e^{-2s\alpha})\cdot\nabla\varphi]\varphi\, dxdt   \\
	\noalign{\smallskip}\dis
	&=&  {1\over2}\iint_Q \Delta(\zeta e^{-2s\alpha})|\varphi|^2\, dxdt  +
	\iint_Q \nabla(\zeta e^{-2s\alpha}\varphi)\cdot\nabla\varphi\, dxdt\\
	\noalign{\smallskip}\dis
	&:=& B_1 + B_2.
\end{eqnarray*}

	Now, let us estimate the terms $B_1$ and $B_2$. For that, we use that
$$	
	\Delta(\zeta e^{-2s\alpha}):=e^{-2s\alpha}\left\{\Delta \zeta+4s\lambda\xi \left[\nabla\zeta\cdot\nabla\eta\right]
	+ 2s\lambda\xi\zeta\left[\lambda|\nabla\eta|^2(2s\xi+1)+\Delta\eta\right]\right\},
$$	
	to see that
\begin{equation*}
	B_1\le C\,s^2\lambda^2 \int_0^T\!\!\!\!\int_{\mathcal{O}_{\om_3}(t)} \!\!\!\!\!\!e^{-2s\alpha}\xi^2 |\varphi|^2\, dxdt.
\end{equation*}

	Since $A\cdot\nabla\psi:=\nabla\cdot(A\psi)-\psi\nabla\cdot A\in H^{-1}(\Om)$, the solution for $\eqref{adj:BBM}_1$ satisfies the following weak formulation
\[
	(\varphi,w)+(\nabla \varphi, \nabla w)= -(A \psi, \nabla w)- (\psi\nabla\cdot A, w)\quad \forall w\in H_0^1(\Om).
\]

	Using the previous formulation with $w=\zeta e^{-2s\alpha}\varphi$, we obtain 
\begin{eqnarray*}
	B_2&=&-\iint_Q \zeta e^{-2s\alpha}|\varphi|^2\, dxdt
	-\iint_Q \left(\nabla\cdot A\right)\zeta e^{-2s\alpha}\varphi\psi\, dxdt
	-\iint_Q \psi \left[A\cdot\nabla(\zeta e^{-2s\alpha}\varphi)\right]\, dxdt\\
	&:=&B_2^1+B_2^2+B_2^3.
\end{eqnarray*}

Now, for $B_2^1$, we easily deduce that
\begin{equation*}
		|B_2^1| \le C\,s^2\lambda^2 \int_0^T\!\!\int_{\mathcal{O}_{\om_3}(t)} \!\!\!\!\!\!e^{-2s\alpha}\xi^2 |\varphi|^2\, dxdt.
\end{equation*}

For $B_2^2$, we notice that, for every $\delta >0$, we obtain
\begin{equation*}
	\begin{alignedat}{2}
		|B_2^2| \le&~ \delta s\lambda^2\iint_Q  e^{-2s\alpha}\xi|\psi|^2\, dxdt
		+C_{\delta}\,s^2\lambda^2 \int_0^T\!\!\int_{\mathcal{O}_{\om_3}(t)} \!\!\!\!\!\!e^{-2s\alpha}\xi^2 |\varphi|^2\, dxdt.
	\end{alignedat}
\end{equation*}
	
	Since $\nabla(\zeta e^{-2s\alpha})=e^{-2s\alpha}(\nabla \zeta+2s\lambda \xi\zeta \nabla\eta)$, for every $\delta >0$, 
	we have that
\begin{equation*}
	\begin{alignedat}{2}
	B_2^3 
	\le&~\varepsilon \left(s^2\lambda^2\iint_Q  e^{-2s\alpha}\xi^2|\varphi|^2\, dxdt
	+\iint_Q  e^{-2s\alpha}|\nabla\varphi|^2\, dxdt\right) + C_\delta \,s^2\lambda^2 \int_0^T\!\!\!\!\int_{\mathcal{O}_{\om_3}(t)} \!\!\!\!\!\!e^{-2s\alpha}\xi^2 |\psi|^2\, dxdt.
	\end{alignedat}
\end{equation*}
	
 	This way, we get
\begin{equation}\label{est:elliptic_BBM_no_grad}
\begin{alignedat}{2}
	&s\lambda ^2  \iint_Qe^{-2s\alpha } \xi |  \psi |^2 dxdt+s^2\lambda^2 \iint_Q e^{-2s\alpha}\xi^2 |\varphi|^2\, dxdt+\iint_Q e^{-2s\alpha}|\nabla\varphi|^2\, dxdt\\
	& \leq
	C \left(  s^2\lambda ^2\int_0^T\!\!\int_{\mathcal{O}_{\om_3}(t)} e^{-2s\alpha } \xi^2 |  \psi |^2 dxdt
	+\,s^2\lambda^2 \int_0^T\!\!\int_{\mathcal{O}_{\om_3}(t)} \!\!\!\!\!\!e^{-2s\alpha}\xi^2 |\varphi|^2\, dxdt  \right ).
\end{alignedat}
\end{equation}

	Finally, to estimate the local integral of $\varphi$ in the right-hand side of \eqref{est:elliptic_BBM_no_grad}, we need to have some global integral in $\varphi_t$ in the left-hand side. 
	For that, we first take the time derivative in $\eqref{adj:BBM}_1$, and the fact that 
	$(A\cdot\nabla\psi)_t=\nabla\cdot(A_t\psi +A\psi_t)-\psi_t\nabla\cdot A-\psi\nabla \cdot A_t$, and use the equation $\eqref{adj:BBM}_2$
	 to see that $\varphi_t$ solves the following  elliptic equation 
\begin{equation}\label{adj:BBMvarphi}
	\left |   
		\begin{array}{lcl}
			\varphi_t - \Delta \varphi_t = \nabla\cdot(A_t\psi -A\varphi)+\varphi \nabla\cdot A-\psi\nabla \cdot A_t 
			&     \mbox{in}& Q,  \\
			\noalign{\smallskip}\dis
			\varphi_t= 0     &    \mbox{on} &   \Sigma.
		\end{array}
	\right. 
\end{equation}
	From  \eqref{est:elliptic_BBM_no_grad} and energy estimates  for \eqref{adj:BBMvarphi}, it is not difficult to see that 
\begin{equation}\label{eq:BBMenergy} 
	\begin{alignedat}{2}
	&\int_Q [s\lambda ^2 \xi^*|\nabla \varphi_t|^2 + s\lambda ^2\xi^* |\varphi_t|^2  ]e^{-2s\alpha^*} dxdt \\
 	&\le C \left(  s^2\lambda ^2\int_0^T\!\!\!\!\int_{\mathcal{O}_{\om_3}(t)}e^{-2s\alpha } \xi^2 |  \psi |^2 dxdt
	+\,s^2\lambda^2 \int_0^T\!\!\!\!\int_{\mathcal{O}_{\om_3}(t)} \!\!\!\!\!\!e^{-2s\alpha}\xi^2 |\varphi|^2\, dxdt  \right ).  
	\end{alignedat}
\end{equation}

	Combining  \eqref{est:elliptic_BBM_no_grad} and  \eqref{eq:BBMenergy}, we get 
\begin{equation}\label{est:elliptic_BBM_with_temp}
\begin{alignedat}{2}
		&\iint_Qe^{-2s\alpha} [ |\nabla \varphi|^2    +  s^2\lambda ^2\xi ^2 |\varphi|^2 ]dxdt 
		+  s\lambda ^2  \iint_Q\xi |  \psi |^2 e^{-2s\alpha}dxdt \\
		&+\int_Q [s\lambda ^2 \xi^*|\nabla \varphi_t|^2 + s\lambda ^2\xi^* |\varphi_t|^2  ]e^{-2s\alpha^*} dxdt\\
		\le &~C \left(  s^2\lambda ^2\int_0^T\!\!\!\!\int_{\mathcal{O}_{\om_3}(t)} e^{-2s\alpha } \xi^2 |  \psi |^2 dxdt
	+\,s^2\lambda^2 \int_0^T\!\!\!\!\int_{\mathcal{O}_{\om_3}(t)}\!\!\!\!\!\!e^{-2s\alpha}\xi^2 |\varphi|^2\, dxdt  \right ).
	\end{alignedat}
\end{equation}
 
	 To finish the proof we estimate the local integral of $\varphi$ by using that $\eqref{adj:BBM}_2$
	 and similar arguments to the ones in the proof of Theorem \ref{thm:Carl_BZK}.

\end{proof}


 \noindent\textbf{Acknowledgements}.
We would like to thanks Prof. G. Lebeau for several discussion about this work. Also, the authors thanks Prof. S. Ervedoza for point out references \cite{Macia, Debayan}.



\appendix


\section{Carleman inequality for the Laplace operator}\label{App:mov_lapla}

	In this section we give the main ideas for the proof of Lemma \ref{lemma:elliptic}. We borrow the ideas  from \cite{CSZR,FI}. 
	
	First, we fix a time $t\in[0,T]$,  set the weight $\gamma(x,t) := e^{\lambda\eta(x,t)}$ and consider the function $w(x,t):=e^{\tau\gamma(x,t)}z(x,t)$.
	 	 
	 Then, we have
\[
	e^{\tau\gamma } \Delta(e^{-\tau\gamma } w ) = \mathcal{M}_1 w - \mathcal{M}_2w,
\]
where 
\[
	\mathcal{M}_1w:=\Delta w+\tau^2|\nabla\gamma|^2w
\]
	is the self-adjoint part of the operator and 
\[
	\mathcal{M}_2w:=2\tau\nabla\gamma\cdot\nabla w+\tau\Delta\gamma w
\]
	is its skew-adjoint part, respectively. 
	
	It follows that 
\be\label{B:decomp}
	\|e^{\tau\gamma}\Delta z\|^2_{L^2(\Om)}=\|\mathcal{M}_1 w\|^2_{L^2(\Om)}+\|\mathcal{M}_ 2w\|^2_{L^2(\Om)}-2(\mathcal{M}_1 w,\mathcal{M}_2 w)_{L^2(\Om)}.
\ee

	The idea of the proof is to analyze the term $(\mathcal{M}_1 w,\mathcal{M}_2 w)_{L^2(\Om)}$. Indeed, a straightforward computation gives
\begin{align*}
	-2(\mathcal{M}_1 w,\mathcal{M}_2 w)_{L^2(\Om)} =&~4\tau  \sum\limits_{i,j=1}^N\ii  \partial^2_{ij}  \gamma \partial _j w\partial _i w\,dx
	-2\tau  \int_{\partial \Omega } {\partial \gamma\over\partial\nu} \left|{\partial w\over\partial \nu}\right|^2 d\Gamma \\
	&-\ii|w|^2 [\tau\Delta ^2\gamma-2\tau^3 \nabla \gamma \cdot \nabla (|\nabla \gamma |^2)]\,dx.
\end{align*}
	Consequently, \eqref{B:decomp} may be rewritten as
\begin{align*}
	\|e^{\tau \gamma} \Delta z\|^2_{L^2(\Om)} =&~\|\mathcal{M}_1 w\|^2_{L^2(\Om)} + \|\mathcal{M}_2 w\| ^2_{L^2(\Om)}
	+4\tau  \sum\limits_{i,j=1}^N\ii \partial^2_{ij} \gamma \partial _j w\partial _i w\,dx\\
	&-2\tau  \int_{\partial \Omega } {\partial \gamma\over\partial\nu} \left|{\partial w\over\partial \nu}\right|^2 d\Gamma
	-\ii|w|^2 [\tau\Delta ^2\gamma-2\tau^3 \nabla \gamma \cdot \nabla (|\nabla \gamma |^2)]\,dx.
\end{align*}

To finish the proof, we need the following claim.
\begin{claim}\label{claim:1} 
	There exist  constants $\lambda _3 >0$, $\tau_3>0$ and  $K\in (0,1)$, independents of $t\in[0,T]$, such that 
	for all $\lambda \ge \lambda _3$ and all $\tau\ge \tau_3$ we have 
 \begin{align}\label{eq:claim}
	\!\!\!\ii |w|^2 [ 2\tau^3 \nabla \gamma \cdot \nabla(|\nabla\gamma |^2)-\tau \Delta ^2 \gamma]dx
	+\!{1\over K}\lambda\tau^3\!\!\int_{\mathcal{O}_{\omega_1}(t)}\!\!\!\!\!\!\!\lambda^3\gamma^3|w|^2dx\ge K\lambda s^3\!\!\ii\!\lambda^3\gamma^3|w|^2dx,
	~\forall t \in [0,T].
\end{align}
\end{claim}

In fact, from this claim and since ${\partial\gamma\over\partial\nu}=\lambda \gamma{\partial\eta\over\partial\nu}\le 0$ on $\partial \Omega\times[0,T]$, see \eqref{P5}, we have
\begin{equation}\label{eq:tgther}
\begin{alignedat}{2}
	&\|\mathcal{M}_1w\|^2_{L^2(\Om)}+\|\mathcal{M}_2 w\|^2_{L^2(\Om)}+K\lambda^4\ii(\tau\gamma)^3|w|^2\,dx\\
	&\le\|e^{\tau\gamma} \Delta z\|^2_{L^2(\Om)} 
	- 4\tau\sum\limits_{i,j=1}^N\ii \partial _j\partial _i \gamma \partial _j w  \partial _i w\,dx
	+K^{-1}\lambda^4\int_{\mathcal{O}_{\omega_1}(t)}(\tau\gamma)^3 |w|^2\,dx.
\end{alignedat}
\end{equation}

The  finishes the proof of Lemma \ref{lemma:elliptic} since one has that
\be
\label{est:zero_order}
\tau \lambda^2 \ii \gamma   |\nabla w|^2  \le C\left(  \tau^{-1} \|\mathcal{M}_1 w\|^2 +  \tau^3 \lambda^4  \ii \gamma ^3 |w|^2 \right) .
\ee

\begin{proof}[Proof of Claim \ref{claim:1}]
%
The claim follows immediately from  the existence of  $\tau_0, \lambda_0 >0$ such that the estimates
\begin{eqnarray*}
  2\tau^3 \nabla \gamma \cdot \nabla 
|\nabla \gamma |^2  - \tau\Delta ^2 \gamma &\ge & A \tau^3 \lambda^4  \gamma^3, \quad t\in [0,T], \ x\in \overline{\Omega} \setminus X(\omega _1,t,0)\\   
| 2\tau^3 \nabla \gamma \cdot \nabla 
|\nabla \gamma |^2 - \tau\Delta ^2 \gamma   |
 &\le &  3 A^{-1} \tau^3 \lambda^4  \gamma^3, \quad t\in [0,T], \ x\in  X(\omega _1,t,0). 
\end{eqnarray*}
holds for every  $\tau \ge \tau_0$ and for all $\lambda \ge \lambda _0$
\end{proof}


\section{Carleman inequality for the Laplace operator in $H^{-1}(\Om)$}\label{App:H-1}



 We give the sketch of the proof of Lemma \ref{carleman:elliptic_H-1}, which is inspired by the arguments in \cite{imanovilovpuel}.

	For every $t \in [0,T]$, we set the function $\gamma(x,t) = e^{\lambda \eta(x,t)}$
	and consider $w(x,t)=e^{\tau\gamma(x,t)}z(x,t)$.
	
	We have the following decomposition
\[
	e^{\tau\gamma} \Delta z  = e^{\tau\gamma } \Delta(e^{-\tau\gamma } w ) = [\Delta w+\tau^2|\nabla\gamma|^2w] - [2\tau\nabla\gamma\cdot\nabla w+\tau\Delta\gamma w]=e^{\tau\gamma }g
	+e^{\tau\gamma }\nabla \cdot G=e^{\tau\gamma }\tilde g+\nabla \cdot(e^{\tau\gamma } G),
\]
	where $\tilde g=g-\tau\nabla \gamma\cdot G$.
	
		Multiplying the previous equation by $w$ and integrating by parts, one can see that:
\[
	\int_\Om |\nabla w|^2\,dx+\tau^2\int_\Om |\nabla\gamma|^2|w|^2\,dx=\int_\Om e^{\tau\gamma }\tilde g w\,dx
	- \int_\Om e^{\tau\gamma }G\cdot\nabla w\,dx
\]	
	which, together with the properties of the weigh function,  gives the result. 

\bibliographystyle{siam}

\end{document}